\documentclass[11pt, reqno, psamsfonts]{amsart}
\pdfoutput=1

\usepackage{amssymb}
\usepackage{amsthm}
\usepackage{amsmath}
\usepackage{latexsym}
\usepackage[T1]{fontenc}
\usepackage[utf8]{inputenc}
\usepackage[russian, french, english]{babel}

\usepackage{graphicx}
\usepackage{wrapfig}
\usepackage[justification=centering, labelfont=bf]{caption}
\usepackage{mathtools}
\usepackage{tikz}
\usepackage{amsbsy}
\usepackage[inline]{enumitem}
\usepackage{mathrsfs}
\usetikzlibrary{shapes,snakes}
\usetikzlibrary{arrows.meta}
\usetikzlibrary{decorations.pathmorphing}
\usetikzlibrary{patterns}
\usetikzlibrary{calc}
\usepackage{float}
\usepackage{array}
\usepackage{subcaption}
\usepackage{multicol}
\usepackage{stmaryrd}
\usepackage{cancel} 
\usepackage{algorithm}
\usepackage{algpseudocode}
\usepackage{lmodern}
\usepackage{mathabx}
\usepackage{upgreek}
\usepackage{titlesec}
\usepackage{bbm}
\usepackage[spacing=true,kerning=true,babel=true,tracking=true]{microtype}
\usepackage[shortcuts]{extdash}
\usepackage[foot]{amsaddr}
\usepackage[left=1in,right=1in,top=1in,bottom=1in,bindingoffset=0cm]{geometry}
\usepackage{bm}
\usepackage{centernot}
\usepackage{mdframed}
\usepackage{hyperref}
\usepackage{mathdots}
\usepackage{xspace}
\hypersetup{
    colorlinks=true,
    linkcolor=blue,
    filecolor=magenta,      
    urlcolor=red,
    citecolor=gray,
}
\allowdisplaybreaks

\usepackage[skins]{tcolorbox}

\usepackage[backend=biber, style=alphabetic, sorting=nyt, maxnames=100,backref=true]{biblatex}
\title{Borel Vizing's Theorem for Graphs \\ of Subexponential Growth}
\date{}
\author{\lsstyle Anton~Bernshteyn}
\email{bernshteyn@math.ucla.edu}
\author{\lsstyle Abhishek~Dhawan}
\email{adhawan2@illinois.edu}
\address{\normalfont{(AB) Department of Mathematics, University of California, Los Angeles, CA, USA}}
\address{\normalfont{(AD) Department of Mathematics, University of Illinois at Urbana--Champaign, IL, USA}}
\thanks{AB's research is partially supported by the NSF grant DMS-2045412 and the NSF CAREER grant DMS-2239187.}

\newtheoremstyle{bfnote}%
{}{}%
{\slshape}{}%
{\bfseries}{\bfseries.}%
{ }%
{\thmname{#1}\thmnumber{ #2}\thmnote{ \ep{\normalfont{}#3}}}

\newtheoremstyle{claim}%
{}{}%
{\slshape}{}%
{\itshape}{.}%
{ }%
{\thmname{#1}\thmnumber{ #2}\thmnote{ \ep{\normalfont{}#3}}}

\theoremstyle{bfnote}
\newtheorem{theo}{Theorem}[section]
\newtheorem*{theo*}{Theorem}

\newtheorem{Lemma}[theo]{Lemma}

\newtheorem{conj}[theo]{Conjecture}

\newtheorem*{corl*}{Corollary}

\theoremstyle{definition}
\newtheorem{defn}[theo]{Definition}
\newtheorem*{defn*}{Definition}

\newtheorem*{exmp*}{Example}

\theoremstyle{remark}
\newtheorem*{ques*}{Question}
\newtheorem*{remk*}{Remark}

\theoremstyle{claim}

\newcounter{ForClaims}[section]

\newtheorem*{claim*}{Claim}

\makeatletter
\newcommand{\neutralize}[1]{\expandafter\let\csname c@#1\endcsname\count@}
\makeatother

\newcommand{\set}[1]{\{#1\}}
\newcommand{\N}{{\mathbb{N}}}

\renewcommand{\epsilon}{\varepsilon}

\renewcommand{\phi}{\varphi}
\renewcommand{\theta}{\vartheta}
\renewcommand{\leq}{\leqslant}
\renewcommand{\geq}{\geqslant}
\newcommand{\defeq}{\coloneqq}

\newcommand{\bemph}[1]{{\normalfont#1}} 
\newcommand{\ep}[1]{\bemph{(}#1\bemph{)}} 

\newcommand{\pto}{\dashrightarrow}
\newcommand{\emphdef}[1]{\textbf{\textit{{#1}}}}

\newcommand{\dom}{\mathsf{dom}}

\numberwithin{equation}{section}
\newcommand{\emphd}[1]{\emphdef{#1}}

\newcommand{\poly}{\mathsf{poly}}

\newcommand{\LOCAL}{$\mathsf{LOCAL}$\xspace}

\newcommand{\rest}[2]{{{#1}\vert_{#2}}}

\titleformat{\section}[block]{\scshape}{\thesection.}{1ex}{}
\titleformat{\subsection}[block]{\bfseries}{\thesubsection.}{1ex}{}
\titleformat{\subsection}[block]{\bfseries}{\thesubsection.}{1ex}{}
\titleformat{\subsubsection}[runin]{\itshape}{\bfseries\upshape\thesubsubsection.}{1ex}{}[.---]

\titlespacing*{\section}{0pt}{*3}{*1}
\titlespacing*{\subsection}{0pt}{*3}{*1}
\titlespacing*{\subsubsection}{0pt}{*1.5}{*0}

\renewbibmacro{in:}{}

\renewbibmacro*{volume+number+eid}{%
	\printfield{volume}%
	\setunit*{\addnbspace}
	\printfield{number}%
	\setunit{\addcomma\space}%
	\printfield{eid}}

\DeclareFieldFormat[article]{volume}{\textbf{#1}\space}
\DeclareFieldFormat[article]{number}{\mkbibparens{#1}}

\DeclareFieldFormat{journaltitle}{#1,}
\DeclareFieldFormat[thesis]{title}{\mkbibemph{#1}\addperiod}
\DeclareFieldFormat[article, unpublished, thesis]{title}{\mkbibemph{#1},}
\DeclareFieldFormat[book]{title}{\mkbibemph{#1}\addperiod}
\DeclareFieldFormat[unpublished]{howpublished}{#1, }

\DeclareFieldFormat{pages}{#1}

\DeclareFieldFormat[article]{series}{Ser.~#1\addcomma}

\setlist{topsep=3pt,itemsep=3pt}

\begin{document}


\maketitle


\begin{abstract}
    We show that every Borel graph $G$ of subexponential growth has a Borel proper edge-coloring with $\Delta(G) + 1$ colors. We deduce this from a stronger result, namely that an $n$-vertex (finite) graph $G$ of subexponential growth can be properly edge-colored using $\Delta(G) + 1$ colors by an $O(\log^\ast n)$-round deterministic distributed algorithm in the \LOCAL model, where the implied constants in the $O(\cdot)$ notation are determined by a bound on the growth rate of $G$.
\end{abstract}

\section{Introduction}

    In this note we study a classical concept in graph theory---namely proper edge-colorings---from the perspective of descriptive set theory. This line of inquiry forms part of the active and growing field of \emphd{descriptive combinatorics}, which was created in the seminal work of Kechris, Solecki, and Todorcevic \cite{KST}. For surveys of this area, see \cite{KechrisMarks} by Kechris and Marks and \cite{Pikh_survey} by Pikhurko. We use standard terminology from graph theory \cite{Diestel,BondyMurty} and from descriptive set theory \cite{KechrisDST,AnushDST}. A graph $G$ consists of a vertex set $V(G)$ and an edge set $E(G) \subseteq [V(G)]^2$, where for a set $X$, we write $[X]^2$ to denote the set of all $2$-element subsets of $X$. We shall mostly be concerned with infinite (in fact, uncountable) graphs in this paper. However, all graphs we shall work with will have finite {maximum degree}, i.e., there will be a uniform finite upper bound on the number of neighbors of every vertex. As usual, $\N$ is the set of all non-negative integers and each $q \in \N$ is identified with the $q$-element set $q = \set{i \in \N \,:\, i < q}$.
    
    \begin{defn}[Edge-colorings and chromatic index]
        Let $G$ be a graph and let $q \in \N$. A \emphd{proper $q$-edge-coloring} of $G$ is a function $\phi \colon E(G) \to q$ such that $\phi(e) \neq \phi(e')$ whenever $e$, $e' \in E(G)$ are distinct edges that share an endpoint. The \emphd{chromatic index} of $G$, denoted by $\chi'(G)$, is the minimum $q \in \N$ such that $G$ has a proper $q$-edge-coloring (if no such $q \in \N$ exists, we set $\chi'(G) \defeq \infty$).
    \end{defn}

    If $G$ is a graph of finite maximum degree $\Delta$, then $\chi'(G) \geq \Delta$, since all edges incident to a vertex of degree $\Delta$ must be colored differently. Vizing famously proved an upper bound on $\chi'(G)$ that is only $1$ larger than this trivial lower bound:
    
    \begin{theo}[{Vizing \cite{Vizing}}]\label{theo:Vizing}
        If $G$ is a graph of finite maximum degree $\Delta$, then $\chi'(G) \leq \Delta + 1$.
    \end{theo}
    
    See \cite[\S{}A.1]{EdgeColoringMonograph} for an English translation of Vizing's original paper and \cites[\S17.2]{BondyMurty}[\S5.3]{Diestel} for modern textbook presentations. 

    We are interested in edge-colorings of Borel graphs:

    \begin{defn}[Borel graphs and their edge-colorings]
        A graph $G$ is \emphd{Borel} if $V(G)$ is a standard Borel space and $E(G)$ is a Borel subset of $[V(G)]^2$. The \emphd{Borel chromatic index} $\chi_\mathsf{B}'(G)$ of $G$ is the smallest $q \in \N$ such that $G$ has a Borel proper $q$-edge-coloring $\phi \colon E(G) \to q$, meaning that $\phi^{-1}(i)$ is a Borel subset of $E(G)$ for all $0 \leq i < q$. If no such $q \in \N$ exists, we set $\chi'_\mathsf{B}(G) \defeq \infty$.
    \end{defn}

    The best general upper bound on $\chi'_\mathsf{B}(G)$ in terms of the maximum degree of $G$ is $2\Delta - 1$:

    \begin{theo}[{Kechris--Solecki--Todorcevic \cite[15]{KST}, Marks \cite[Theorem 1.4]{Marks}}]\label{theo:BorelEdgeColoring}
        Fix an integer $\Delta \in \N$.

        \begin{enumerate}[label=\ep{\itshape\roman*}]
            \item Every Borel graph $G$ of maximum degree $\Delta$ satisfies $\chi'_\mathsf{B}(G) \leq 2\Delta - 1$.

            \item On the other hand, there exists a Borel graph $G$ of maximum degree $\Delta$ with $\chi'_\mathsf{B}(G) = 2\Delta - 1$.
        \end{enumerate}
    \end{theo}    

    It has been a matter of interest to discover whether in some special cases the bound on $\chi'_\mathsf{B}(G)$ given by Theorem~\ref{theo:BorelEdgeColoring} can be lowered, ideally all the way down to Vizing's bound $\Delta + 1$. For example, Greb\'ik \cite{GmeasVizing} recently showed that the Borel chromatic index of any Borel graph can be reduced to $\Delta + 1$ by throwing away a set of vertices of measure $0$ (this result builds on the earlier breakthrough work of Greb\'ik and Pikhurko \cite{GP}). 
%
%
    In a similar vein, Qian and Weilacher \cite{ASIalgorithms} showed that the bound can be reduced to $\Delta + 2$ after removing a meager set of vertices (in the sense of Baire category). Alternatively, instead of discarding a small set of vertices, one may try to prove stronger upper bounds on $\chi'_\mathsf{B}(G)$ under extra assumptions on the structure of $G$; see, e.g., \cite{BWKonig,ASIalgorithms,FelixFinDim} for some instances of this approach. Our main result is of this type: we establish the bound $\chi'_\mathsf{B}(G) \leq \Delta + 1$ for Borel graphs $G$ of subexponential growth.
    
    Given a graph $G$, a vertex $v$, and an integer $R \in \N$, we write $N_G^R[v]$ for the \emphd{closed $R$-neighborhood of $v$}, i.e., the set of all vertices reachable from $v$ by a path of at most $R$ edges.

    \begin{defn}[Subexponential growth]\label{defn:subexp}
    A function $f \colon \N \to \N$ is \emphd{subexponential} if
    \begin{equation}\label{eq:Reps}
        \forall \epsilon > 0,\, \exists R(\epsilon) \in \N \text{ such that } \forall R\geq R(\epsilon),\ f(R) < \exp(\epsilon\,R).
    \end{equation}
    A function $f \colon \N \to \N$ \emphd{bounds the growth} of a graph $G$ if
    \[\forall v \in V(G),\, \forall R \in \N, \  |N_G^R[v]| \leq f(R).\]
    A graph is \emphd{of subexponential growth} if its growth is bounded by a subexponential function.
    \end{defn}

    The main result of this paper is as follows:

    \begin{tcolorbox}
    \begin{theo}[Borel Vizing's theorem for subexponential growth graphs]\label{theo:main_theorem}
        If $G$ is a Borel graph of subexponential growth and of finite maximum degree $\Delta$, then $\chi_\mathsf{B}'(G) \leq \Delta + 1$.
    \end{theo}
    \end{tcolorbox}

    Theorem~\ref{theo:main_theorem} contributes to the growing body of research showing that various combinatorial problems can be solved in a Borel way on graphs of subexponential growth; see, e.g., \cite{ConleyTamuz,CGMPT, Thornton}.
    
    We deduce Theorem~\ref{theo:main_theorem} from a related result in \emphd{distributed computing}---an area of computer science concerned with problems that can be solved efficiently by a decentralized network of processors. The connection between descriptive combinatorics and distributed computing has been discovered by the first named author \cite{BernshteynDistributed} and is actively studied by both descriptive set-theorists and computer scientists \cite{Ber_cont, trees, GRgrids, FelixComput}. It is now understood that efficient distributed algorithms can often be used in a ``black box'' manner to derive results in descriptive combinatorics \cite{BernshteynDistributed}, and sometimes one can also go in the opposite direction \cite{Ber_cont,trees,GRgrids}.

    The relevant model of distributed computation, introduced by  Linial in \cite{Linial}, is called \LOCAL. 
    For an introduction to this subject, see the book \cite{BE} by Barenboim and Elkin. In this model an $n$-vertex (finite) graph $G$ abstracts a communication network where each vertex plays the role of a processor and edges represent communication links. The computation proceeds in synchronous \emphd{rounds}. During each round, the vertices first perform some local computations and then simultaneously broadcast messages to their neighbors. There are no restrictions on the complexity of the local computations or on the length of the messages. After a certain number of rounds, each vertex must output its part of the global solution (for instance, its own color or, in the context of edge-coloring, the colors of the edges incident to it). The efficiency of such an algorithm is measured by the number of communication rounds required, as a function of $n$.
	
	An important feature of the \LOCAL model is that every vertex of $G$ is executing the same algorithm. Therefore, to make this model nontrivial, the vertices must be given a way of breaking symmetry. In the \emphd{deterministic} variant of the model, this is achieved by assigning a unique identifier $\mathsf{ID}(v) \in \set{1,\ldots, n}$ to each vertex $v \in V(G)$.\footnote{Sometimes the range of the identifiers is taken to be $\set{1, \ldots,  n^c}$ for some constant $c \geq 1$, but this does not affect the model significantly.} 
	The identifier assigned to a vertex $v$ is treated as part of $v$'s input; that is, $v$ ``knows'' its own identifier and can communicate this information to its neighbors. When we say that a deterministic \LOCAL algorithm \emphd{solves} a problem $\Pi$ on a given class $\mathbb{G}$ of finite graphs, we mean that its output on any graph from $\mathbb{G}$ is a valid solution to $\Pi$, regardless of the way the identifiers are assigned. The word ``deterministic'' distinguishes this model from the \emph{randomized} version, where the vertices are allowed to generate sequences of random bits. In this paper we shall only be concerned with deterministic algorithms.

    An important observation is that if $u$ and $v$ are two vertices whose graph distance in $G$ is greater than $T$, then no information from $u$ can reach $v$ in fewer than $T$ communication rounds \ep{this explains the name ``\LOCAL''}. Conversely, every $T$-round \LOCAL algorithm can be transformed into one in which every vertex first collects all the data present in its closed $T$-neighborhood and then makes a decision, based on this information alone, about its part of the output (see \cite[\S4.1.2]{BE}).

    The study of \LOCAL algorithms for edge-coloring has a long history. We direct the reader to \cite{CHLPU, GKMU} for thorough surveys and \cite{BDVizing,DaviesSODA} for more recent developments. The case of particular interest to us is when the maximum degree $\Delta$ is treated as a fixed parameter. Similar to the situation with Borel edge-colorings, $2\Delta - 1$ is an important threshold:

    \begin{theo}[{Panconesi--Rizzi \cites{PR}[Theorem 8.5]{BE},  Chang--He--Li--Pettie--Uitto \cite{CHLPU}}]\label{theo:2d-1_alg}
        Fix and integer $\Delta \in \N$.

        \begin{enumerate}[label=\ep{\itshape\roman*}]
            \item\label{item:2d-1_alg} There is a deterministic \LOCAL algorithm that finds a proper $(2\Delta-1)$-edge-coloring of an $n$-vertex graph of maximum degree $\Delta$ in $O(\Delta + \log^\ast n)$ rounds.

            \item On the other hand, every deterministic \LOCAL algorithm for $(2\Delta - 2)$-edge-coloring $n$-vertex graphs of maximum degree $\Delta$ requires at least $\Omega(\log n/\log\Delta)$ rounds.
        \end{enumerate}
    \end{theo}

    In the above statement, $\log^\ast n$ is the \emphd{iterated logarithm} of $n$, i.e., the number of times the logarithm function must be applied to $n$ before the result becomes at most $1$. This is an extremely slow-growing function, asymptotically much smaller than $\log n$ or any finite iterate of $\log n$. Our second main result is that in $O(\log^\ast n)$ rounds, it is possible to find a proper $(\Delta+1)$-edge-coloring of an $n$-vertex graph $G$ with maximum degree $\Delta$ when $G$ belongs to a class of graphs with subexponential growth. More precisely, given $\Delta \in \N$ and a function $f \colon \N \to \N$, we let $\mathbb{G}_{\Delta, f}$ be the class of all finite graphs $G$ of maximum degree at most $\Delta$ and growth bounded by $f$. 

    \begin{tcolorbox}
    \begin{theo}[\LOCAL algorithm for Vizing's theorem on subexponential growth graphs]\label{theo:dist_algo}
        Fix a subexponential function $f \colon \N \to \N$ and $\Delta \in \N$. There is a deterministic \LOCAL algorithm that, given an $n$-vertex graph $G \in \mathbb{G}_{\Delta,f}$, finds a proper $(\Delta+1)$-edge-coloring of $G$ in at most $C_f \log^\ast n$ rounds, where $C_f > 0$ depends only on the function $f$.
    \end{theo}
    \end{tcolorbox}

    Theorem~\ref{theo:main_theorem} is an immediate consequence of Theorem~\ref{theo:dist_algo}, thanks to \cite[Theorem~2.10]{BernshteynDistributed} (in fact, Theorem~\ref{theo:dist_algo} is stronger than what is needed, since, by \cite[Theorem~2.15]{BernshteynDistributed}, having a \emph{randomized} \LOCAL algorithm with the parameters stated in Theorem~\ref{theo:dist_algo} would already suffice to deduce Theorem~\ref{theo:main_theorem}). We explain the details of the derivation of Theorem~\ref{theo:main_theorem} from Theorem~\ref{theo:dist_algo} in \S\ref{sec:implication}. The proof of Theorem~\ref{theo:dist_algo} is presented in \S\ref{sec:algorithm}. The main tool we rely on is a recent result of Christiansen \cite[Theorem 3]{Christ} on small augmenting subgraphs for proper partial $(\Delta+1)$-edge-colorings.

    \section{From Theorem \ref{theo:dist_algo} to Theorem \ref{theo:main_theorem}}\label{sec:implication}

    As mentioned in the introduction, to derive Theorem~\ref{theo:main_theorem} from Theorem~\ref{theo:dist_algo}, we invoke \cite[Theorem 2.10]{BernshteynDistributed}. Roughly speaking, \cite[Theorem 2.10]{BernshteynDistributed} says that if a Borel graph $G$ can be ``approximated'' by finite graphs from a certain class $\mathbb{G}$, then any locally checkable labeling problem that can be solved by an efficient deterministic \LOCAL algorithm on the graphs in $\mathbb{G}$ admits a Borel solution on $G$. 
    The precise statement of \cite[Theorem 2.10]{BernshteynDistributed} is somewhat technical, but since we do not need it in full generality, we shall only state the special case for edge-colorings.  

\begin{theo}[{AB \cite[Theorem 2.10]{BernshteynDistributed}}]\label{theo:dist_algo_to_borel}
    Fix $q \in \N$ and a class of finite graphs $\mathbb{G}$. Suppose that there exists a deterministic \LOCAL algorithm that, given an $n$-vertex graph $H \in \mathbb{G}$, finds a proper $q$-edge-coloring of $H$ in $o(\log n)$ rounds.\footnote{The implicit constants in the $o(\cdot)$ notation may depend on $q$ and the class $\mathbb{G}$.} If $G$ is a Borel graph of finite maximum degree $\Delta$ all of whose finite induced subgraphs are in $\mathbb{G}$, then $\chi'_\mathsf{B}(G) \leq q$.
\end{theo}

    The actual statement of \cite[Theorem 2.10]{BernshteynDistributed} is numerically explicit (i.e., it does not use the asymptotic notation $o(\cdot)$) and does not require all finite induced subgraphs of $G$ to be in $\mathbb{G}$, but the simplified formulation given above will be sufficient for our purposes. See \cite[\S2.B.1]{BernshteynDistributed} for further discussion.

    Assume Theorem~\ref{theo:dist_algo} and let $G$ be a Borel graph of subexponential growth and of finite maximum degree $\Delta$. Let $f \colon \N \to \N$ be a subexponential function that bounds the growth of $G$. 
    Theorem \ref{theo:dist_algo}  yields a constant $C_f > 0$ and a deterministic \LOCAL algorithm for $(\Delta+1)$-edge-coloring $n$-vertex graphs in $\mathbb{G}_{\Delta,f}$ in $C_f\log^\ast n = o(\log n)$ rounds. Since all finite induced subgraphs of $G$ belong to $\mathbb{G}_{f,\Delta}$, we have $\chi'_\mathsf{B}(G) \leq \Delta + 1$ by Theorem~\ref{theo:dist_algo_to_borel}. This finishes the proof of Theorem~\ref{theo:main_theorem}.
    
    \section{Proof of Theorem \ref{theo:dist_algo}}\label{sec:algorithm}


        \begin{defn}[Augmenting subgraphs for partial colorings]\label{defn:aug}
            Let $G$ be a graph of finite maximum degree $\Delta$ and let $\phi \colon E(G) \pto (\Delta + 1)$ be a proper partial $(\Delta + 1)$-edge-coloring with domain $\dom(\phi) \subset E(G)$. A subgraph $H \subseteq G$ is \emphd{$e$-augmenting} for an uncolored edge $e \in E(G) \setminus \dom(\phi)$ if $e \in E(H)$ and there is a proper coloring $\phi' \colon \dom(\phi) \cup \set{e} \to (\Delta + 1)$ that agrees with $\phi$ on the edges in $E(G) \setminus E(H)$; in other words, by modifying the colors of the edges of $H$, it is possible to add $e$ to the set of colored edges. We refer to such a modification operation as \emphd{augmenting} $\phi$ using $H$.
        \end{defn}

        Note that every inclusion-minimal $e$-augmenting subgraph is connected. We will use the following recent result of Christiansen:

\begin{theo}[{Christiansen \cite[Theorem 3]{Christ}}]\label{theo:small_aug}
    There is a constant $C \geq 1$ such that if $G$ is an $n$-vertex graph of maximum degree $\Delta$ and $\phi$ is a proper partial $(\Delta + 1)$-edge-coloring of $G$, then for each uncolored edge $e$, there is an 
    $e$-augmenting subgraph with at most $C\Delta^{7}\log n$ edges. 
\end{theo}

    We remark that the bound $C\Delta^{7}\log n$ in Theorem~\ref{theo:small_aug} is best possible as far as the dependence on $n$ is concerned, since Chang, He, Li, Pettie, and Uitto \cite{CHLPU} showed that, in general, given an uncolored edge $e$, there may not exist an $e$-augmenting subgraph of diameter less than $\Omega(\Delta \log(n/\Delta))$. Christiansen's work followed a sequence of earlier, weaker bounds: $O(n)$ by Vizing \cite{Vizing} (see also \cites[94]{Bollobas}{RD}{MG}), $\poly(\Delta)\sqrt{n}$ by Greb\'ik and Pikhurko \cite{GP} (this bound is not explicitly stated in \cite{GP}; see \cite[\S3]{VizingChain} for a sketch of the derivation of the $\poly(\Delta)\sqrt{n}$ bound using the Greb\'ik--Pikhurko approach), and $\poly(\Delta) \log^2 n$ by the first named author \cite{VizingChain}. In \cite{BDVizing}, the authors gave an alternative proof of the bound $\poly(\Delta)\log n$ and an efficient algorithm to find a small augmenting subgraph.

    Now we are ready to start the proof of Theorem~\ref{theo:dist_algo}. For the remainder of this section, fix a subexponential function $f \colon \N \to \N$, integers $\Delta$, $n \in \N$, and an $n$-vertex graph $G \in \mathbb{G}_{f, \Delta}$. Since the maximum degree of $G$ cannot exceed $f(1)$, we may assume that $\Delta \leq f(1)$. We may also assume that $\Delta \geq 3$, since otherwise we are done by Theorem~\ref{theo:2d-1_alg}\ref{item:2d-1_alg}.

    Define $\epsilon \defeq 1/(3Cf(1)^7)$ and $R \defeq \max\set{R(\epsilon), \,\lceil\epsilon^{-1}\rceil,\, 3}$, where $C$ is the constant from Theorem~\ref{theo:small_aug} and $R(\epsilon)$ is given by \eqref{eq:Reps}. 
    The key observation is that in $G$, we can find augmenting subgraphs whose number of edges is \emph{independent of $n$}: 

\begin{Lemma}\label{lemma:small_aug_subexp}
    Let $\phi$ be a proper partial $(\Delta + 1)$-edge-coloring of $G$.
    Then for each uncolored edge $e = xy$, there exists an $e$-augmenting subgraph with at most $R$ edges.
\end{Lemma}

\begin{proof}
    Let $H \defeq G\big[N_G^R[x] \cup N_G^R[y]\big]$. By the definition of $R$, we have \[|V(H)| \,\leq\, 2f(R) \,<\, 2\exp(\epsilon R) \,<\, \exp(2\epsilon R),\] where in the last inequality we use that $\epsilon R \geq 1$. Consider the partial edge-coloring $\psi \defeq \rest{\phi}{E(H)}$ of $H$. By Theorem \ref{theo:small_aug}, there is a subgraph $H'$ of $H$ such that $H'$ is $e$-augmenting for $\psi$ and
    \[
        |E(H')| \,\leq\, C\Delta^{7} \log |V(H)| \,\leq\, C\Delta^{7}\, 2\epsilon R \,\leq\, 2C(f(1))^{7}\,\epsilon R \,=\, \frac{2R}{3} \,\leq\, R - 1.\]
    We may choose $H'$ to be inclusion-minimal, hence connected, and then it follows that every vertex of $H'$ is reachable from $x$ by a path of at most $R-1$ edges. Let $\psi' \colon \dom(\psi) \cup \set{e} \to (\Delta + 1)$ be 
    a proper coloring obtained by augmenting $\psi$ using $H'$
    and define $\phi' \colon \dom(\phi) \cup \set{e} \to (\Delta + 1)$ by
    \[
        \phi'(h) \,\defeq\, \begin{cases}
            \psi'(h) &\text{if } h \in E(H);\\
            \phi(h) &\text{if } h \in E(G) \setminus E(H).
        \end{cases}
    \]
    We claim that $\phi'$ is proper. Otherwise, since both $\psi'$ and $\phi$ are proper and $\psi'$ agrees with $\phi$ on the edges in $E(H) \setminus E(H')$, there are vertices $u$, $v$, $w \in V(G)$  such that $uv \in E(H')$ and $vw \in E(G) \setminus E(H)$. This implies that $v$ is reachable from $x$ by a path of length at most $R - 1$, so both $v$ and $w$ belong to $N_G^R[x]$, and hence $vw \in E(H)$; a contradiction. 
    Since $\phi$ and $\phi'$ only differ on the edges of $H'$, it follows that, viewed as a subgraph of $G$, $H'$ is  $e$-augmenting for $\phi$, and we are done. 
\end{proof}

    
    Recall that a \emphd{proper $q$-vertex-coloring} of a graph $G$ is a mapping $\phi \colon V(G) \to q$ such that $\phi(u) \neq \phi(v)$ for all $uv \in E(G)$. We shall use a classical result of Goldberg, Plotkin, and Shannon: 

\begin{theo}[{Goldberg--Plotkin--Shannon \cites{GPSh}[Corollary 3.15]{BE}}]\label{theo:dist_linial}
    There exists a deterministic \LOCAL algorithm that, given an $n$-vertex graph $G$ of maximum degree $\Delta$, computes a proper $(\Delta+1)$-vertex-coloring of $G$ in $O(\Delta^2) + \log^\ast n$ rounds.
\end{theo}


    Now we describe our $(\Delta+1)$-edge-coloring algorithm. Consider the graph $G^\ast$ with vertex set $V(G^\ast) \defeq E(G)$ in which two distinct edges of $G$ are adjacent if and only if they are joined by a path of length at most $2R$ in $G$. 
    A single communication round in the \LOCAL model on $G^\ast$ can be simulated by $O(R)$ communication rounds in the \LOCAL model on $G$. Hence, since $\Delta(G^\ast) < (2\Delta)^{2R}$, we can use Theorem \ref{theo:dist_linial} to compute a $(2\Delta)^{2R}$-vertex-coloring $\psi$ of $G^\ast$ in $O(R\log^*n + R(2\Delta)^{4R})$ rounds. 
        Set $q \defeq (2\Delta)^{2R}$ and for each $0 \leq i < q$, let
        \[
            C_i \,\defeq\, \set{e \in E(G)\,:\, \psi(e) = i}.
        \]
        Next we iteratively compute a sequence of proper partial $(\Delta+1)$-edge-colorings $\phi_0$, \ldots, $\phi_q$, where each $\phi_i$ is defined on the edges in $C_0 \cup \ldots \cup C_{i-1}$, as follows. Start with $\phi_0$ being the empty coloring. Once $\phi_{i}$ has been computed, the endpoints of every edge $e \in C_i$ survey their closed $R$-neighborhoods and arbitrarily pick a connected $e$-augmenting subgraph $H_e$ with at most $R$ edges. (Such $H_e$ exists by Lemma~\ref{lemma:small_aug_subexp}.) Since the edges in $C_i$  are not adjacent in $G^\ast$, the graphs $H_e$ for distinct $e \in C_i$ must be vertex-disjoint. Thus, we can augment the coloring $\phi_{i-1}$ using all the graphs $\set{H_e \,:\, e \in C_i}$ simultaneously without creating any conflicts and let $\phi_{i+1}$ be the resulting coloring. Note that given $\phi_{i}$, we compute $\phi_{i+1}$ in at most $O(R)$ rounds (since no communication at distances greater than $R$ is necessary). The final coloring $\phi \defeq \phi_q$ is the desired proper $(\Delta+1)$-edge-coloring of $G$. The total number of rounds needed to compute it is
        \[
            O(R \log^\ast n + R(2\Delta)^{4R} + Rq) \,=\, O(R \log^\ast n + R(2\Delta)^{4R} + R(2\Delta)^{2R}). 
        \]
%
%
Since $\Delta \leq f(1)$, this is at most $C_f\log^* n$ for some constant $C_f$ depending only on $f$, as desired.

\printbibliography

\end{document}